\documentclass[10pt]{amsart}
\usepackage{amssymb,amsfonts,graphicx}
\usepackage{amsmath,amsthm}
\usepackage{chngcntr}
\usepackage{apptools}
\usepackage[utf8]{inputenc}

\usepackage{graphicx,amssymb,latexsym}
\input xy
\usepackage{color}
\usepackage[color,matrix,arrow,all]{xy}
\usepackage[all]{xy}
\usepackage{enumerate}
\xyoption{all}
\usepackage{amscd}
\usepackage{color}
\usepackage{tikz-cd}

\newcommand{\lra}{{\longrightarrow }}
\newcommand{\cO}{{\mathcal O}}
\newcommand{\cM}{{\mathcal M}}

\newcommand{\cA}{{\mathcal A}}

\newtheorem{thm}{Theorem}[section]
\newtheorem{cor}[thm]{Corollary}
\newtheorem{lem}[thm]{Lemma}
\newtheorem{prop}[thm]{Proposition}
\newtheorem{question}[thm]{Question}
\theoremstyle{definition}
\newtheorem{defn}[thm]{Definition}
\newtheorem{rem}[thm]{Remark}

\numberwithin{equation}{section}

\newcommand{\ZZ}{\mathbb Z}
\newcommand{\CC}{\mathbb C}
\newcommand{\PP}{\mathbb P}
\newcommand{\FF}{\mathbb F}

\newcommand{\ra}{\rightarrow}

\newcommand{\cH}{\mathcal{H}}

\newcommand{\tC}{\widetilde{C}}

\newcommand{\cP}{\mathcal{P}}
\newcommand{\cR}{\mathcal{R}}

\newcommand{\cRH}{\mathcal{RH}}
\newcommand{\s}{\sigma}

\DeclareMathOperator{\Aut}{{Aut}}

\DeclareMathOperator{\Ker}{Ker}
\DeclareMathOperator{\Fix}{Fix}

\DeclareMathOperator{\Nm}{{Nm}}

\DeclareMathOperator{\Ima}{{Im}}

\DeclareMathOperator{\Sing}{Sing}

\DeclareMathOperator{\im}{Im}

\DeclareMathOperator{\End}{{End}}
\DeclareMathOperator{\id}{{id}}

\title{Klein coverings over hyperelliptic genus 3 curves}

\author{Pawe\l{} Bor\'owka}
\address{Pawe\l{} Bor\'owka \newline Institute of Mathematics, Jagiellonian University in Krak\'ow\\ ul. prof. Stanisława Łojasiewicza 6, 
30-348 Kraków, Poland}
\email{pawel.borowka@uj.edu.pl}

\author{Angela Ortega}
\address{Angela Ortega
\newline Institut für Mathematik, Humboldt-Universität,
Unter den Linden 6, D-10099 Berlin, Germany}
\email{ortega@math.hu-berlin.de}

\begin{document}

\begin{abstract}
We characterize the moduli space of \'etale Klein coverings (i.e. Galois with deck group $\mathbb{Z}_2^2$) of hyperelliptic curves of genus 3. We prove that the Prym map on each component is injective. As an application, we show that the Prym map of \'etale Klein coverings of genus 3 curves is generically finite.
\end{abstract}
\maketitle

\section*{Introduction}

Let $H$ be a hyperelliptic smooth projective curve over $\CC$ of genus 3. We consider Klein coverings $f:\tC \ra H$, that is, 
$\tC$ admits an action by a group of fixed point free automorphisms $V_4$ isomorphic to $\ZZ_2 \times \ZZ_2$ such that $H = \tC / V_4$. 
One can associate to such a covering a polarised abelian sixfold in the following way. Let $\Nm_f : J\tC \ra JH$ be the norm map, sending the divisor class $[\sum_i n_ip_i]$ to $[\sum_i n_if(p_i)]$,
with $p_i \in \tC$ and $n_i\in \ZZ$. The Prym variety $P(f)$ of the covering $f$ is defined as the connected component of 0 of the kernel of $\Nm_f$. It is an abelian subvariety of $J\tC$ of dimension the difference of the genera of the curves, which in this case is 6 and the restriction of the principal polarisation on $J\tC$ induces a non-principal one on $P(f)$. 

Prym varieties of étale double covers have been extensively studied in the past and more recently it has been proven that 
the Prym map is injective for double covers  ramified in at least six points (\cite{NO22}). For étale cyclic coverings over a hyperelliptic curve, the (generic) injectivity 
of the Prym has been shown for infinitely many degrees (\cite{NOS24}, 
\cite{NOPS25}). There are also new results on the fiber of the Prym map for non-cyclic coverings in low genera (\cite{S25}).

Klein coverings can be defined from the curve $H$ by choosing a subgroup 
$\langle \eta, \xi\rangle $ of order four of the group $JH[2]$ of 2-torsion points in $JH$.  
We distinguish two types of Klein coverings depending whether the subgroup is isotropic or not with respect to the Weil pairing on  $JH[2]$. 
Let $\cRH_3^{iso}$, respectively  $\cRH_3^{ni}$, denote the moduli space parametrising isotropic, respectively  non-isotropic, Klein 
coverings on hyperelliptic curves of genus 3. 
Each of these moduli spaces consists of two irreducible components (see Lemma \ref{4components}). 

Consider now the Prym map that sends $[f:\tC \ra H] \in \cRH_3^{V_4}=\cRH_3^{iso} \sqcup \cRH_3^{ni} $ to the 
polarised Prym variety $(P,\Xi)$. In this article we continue the study of the injectivity of the Prym map for Klein  
coverings of genus $3$ curves, which was started in \cite{BO24}, where we prove the injectivity on the locus of hyperelliptic coverings $f:\tC \ra H $ (i.e. $\tC$ is also hyperelliptic) 
corresponding to the case I.1 below.  Let us denote by
$\cA_g^{\delta}$ the moduli space of polarised abelian varieties of dimension $g$ with polarisation type $\delta$.
The main theorem of the paper is the following.

\begin{thm}
The Prym maps 
$$
\cP^{iso}_3 :  \cRH_3^{iso} \ra  \cA_6^{(1,1,1,2,2,4)} \qquad \cP^{ni}_3 :  \cRH_3^{ni} \ra  \cA_6^{(1,1,1,1,4,4)}
$$
are injective on each irreducible component of the source moduli space. 

\end{thm}

Actually, the injectivity also holds on the entire moduli space $\cRH_3^{V_4}$ since the dimensions of the factors appearing in the isotypical decomposition of the associated Prym variety,  depends on the type of the Klein covering and it is different for each irreducible component. 

The general idea of the proof, with some variations in each case, is the following. A Prym variety $(P,\Xi)$ in the image of the Prym map determines a group of automorphisms of $P$ isomorphic to 
$\ZZ_2^4$ (see Proposition \ref{autoofP}). This gives us the isotypical decomposition of $P$ and one can identify the Jacobians by means of the type of the restricted polarisation and using the action of the involutions on the kernel of the polarisation.  With this information we can reconstruct the quotients curves and finally recover $\tC$ as a fibered product.  

As an application, we can show the following theorem for Klein coverings over {\it any} curve of genus 3, see Theorem \ref{FullKlein}.
\begin{thm}
    The Klein Prym maps $$
\cP^{iso}_3 :  \cR_3^{iso} \ra  \cA_6^{(1,1,1,2,2,4)} \qquad \cP^{ni}_3 :  \cR_3^{ni} \ra  \cA_6^{(1,1,1,1,4,4)}
$$ are generically finite. Moreover they are of degree $1$ on images of coverings of hyperelliptic curves of types $I.2, II.1, II.2$.
\end{thm}

The structure of the paper is as follows. In Section 1, we state basic definitions and key lemmas and show that the moduli space have 4 components, called I.1, I.2, II.1, II.2. Each one of the Sections 2,\ 3,\ 4 and 5 is devoted to one of these cases and to the proof of the injectivity of the corresponding  Prym map. Finally, Section 6 is devoted to Klein coverings of any genus 3 curve.

\subsection*{Acknowledgments}
The authors would like to thank Anatoli Shatsila for finding a reference that shortened considerably one of the proofs. The first author has been supported by the Polish National Science Centre project number 2024/54/E/ST1/00330. Some results of the paper were obtained during his visit to Humboldt University in Berlin. He would like to thank the university for hospitality.

\section{Preliminaries}
We start by setting up some general facts about Klein coverings.
\subsection{Top-down and bottom up perspectives}
In the paper, we focus on Klein coverings, i.e. \'etale Galois coverings $f:\tC\to H$ with the deck group of transformations isomorphic to to the Klein four-group.
As in \cite{BO24}, we consider the construction from two perspectives. 
We start with the bottom-up perspective. Let $H$ be a curve and  
$\left<\eta,\xi\right>\subset JH[2]$
a Klein four-group. Using \cite{Pardini} we can construct the following commutative tower of curves:
\begin{equation}\label{diag:Klein_cov_d}
\xymatrix@R=.9cm@C=.6cm{
& \tC \ar[dr] \ar[dl]_{h_\eta} \ar[d] & \\
C_{\eta}\ar[dr]_{k_\eta} &  C_{\xi} \ar[d] & C_{\eta +\xi} \ar[dl]\\
&H& 
}
\end{equation}
where $k_\eta:C_\eta\to H$ is a double covering constructed from $(H,\eta)$ and $h_\eta:\tC\to C_\eta$ is constructed from $(C_\eta,k^*_{\eta}(\xi))$. The other coverings are defined analogously.
We call $\tC\to H$ the Klein covering constructed from $(H,\left<\eta,\xi\right>)$.

On the other hand, one can look at the construction from top-down perspective. Consider $\tC$ with a  Klein subgroup of fixed-point free automorphisms $\left<\s,\tau\right>\subset \Aut(\tC)$.
Denoting by $C_\alpha$ the quotient curve $C/\left<\alpha\right>$, we obtain the following commutative diagram of curves
\begin{equation}\label{diag:Klein_cov_t}
\xymatrix@R=.9cm@C=.6cm{
& \tC \ar[dr] \ar[dl]_{h_\s} \ar[d] & \\
C_{\s}\ar[dr]_{k_\s} &  C_{\tau} \ar[d] & C_{\s\tau} \ar[dl]\\
&H=C/\left<\s,\tau\right>& 
}
\end{equation}

Both perspectives describe the same construction, so from now on we assume that the diagrams coincide and in particular $C_\eta=C_\s$. We will also freely change the perspective and we assume that the reader understands from the context and from the notation which perspective we are in.

We recall the following notation which is a generalisation of the notion of complementary abelian subvarieties. Let $A$ be an abelian variety and $A_1, A_2 \ldots, A_n \subseteq A$ be a set of abelian subvarieties with the induced polarisations. We denote the symmetric idempotent of $A_k$ by $\varepsilon_{k} \in \End_{\mathbb{Q}}(A)$ for $1 \leq k \leq n$. We write $A = A_1 \boxplus A_2 \boxplus \ldots \boxplus A_n$ if $\sum_{i=1}^n \varepsilon_i = 1.$ 
Denote by $P(k)$ the Prym variety corresponding to a map $k$. We recall the following basic fact.
\begin{lem}\cite[Prop 5.2.2]{LR}\label{basicisodec}
   The isotypical decomposition of $J\tC$ with respect to $\left<\s,\tau\right>$
is given by 
$$
  J\tC  =   f^*JH \boxplus h_\s^*P(k_\s) \boxplus h_\tau^*P(k_\tau) \boxplus h_{\s\tau}^*P(k_{\s\tau})
$$ 
\end{lem}

In order to simplify the notation, if the pullback map is an embedding, we abuse the notation by skipping the map (for example, we write $JC\subset JC'$) and if it is not an embedding, we mark the subvariety with a star (for example,  we write $JC^*\subset JC'$).

Looking at the construction from bottom-up perspective, we can distinguish two cases of Klein coverings. Recall the commutator map, \cite[\S 6.3]{BL}, that in case of Jacobians and twice the principal polarisation, we call the Weil pairing.
\begin{defn}
We call a Klein covering isotropic if the group $\left<\eta,\xi\right> \subset JC[2]$ defining the covering is isotropic with respect to the Weil pairing. Otherwise, we call the covering non-isotropic.
\end{defn}
Being isotropic or not determine the type of the induced polarisation and affects the decomposition of $J\tC$.
\begin{lem}\label{polonJH}
Let $\langle \eta, \xi \rangle \simeq \ZZ_2 \times \ZZ_2$ be a Klein subgroup of $JH[2]$.
If $\langle \eta, \xi \rangle$ is isotropic, respectively non-isotropic, with respect to the Weil form in $JH[2]$, the 
polarisation type of the restriction $\Xi=\tilde \Theta_{|P}$ 
to the corresponding Prym variety is $(1,\ldots,1,2,2,4,\ldots,4)$, respectively  $(1,\ldots,1,1,4,\ldots,4)$.
\end{lem}
\begin{proof}
Set $G:=\left<\eta,\xi\right>$.
Let $f^*:JH\to J\tC$ be the pullback map. Then $f^*\tilde{\Theta}=4\Theta_H$ and $\ker(f^*)=G$, see \cite[Prop 11.4.3]{BL}.

Consider the quotient map $\pi: JH\to JH/G$. If $G$ is isotropic with respect to the Weil form on $JH[2]$ then, in order to check that  $\pi$ is a polarised isogeny,  by \cite[Prop 6.3.5]{BL} it is enough to show that we can endow $JH$ with a polarisation $2\Theta$ (of type $(2,\ldots,2)$). Since $\pi$ is of degree 4, we find  that $JH/G$ is of type $(1,1,2,\ldots,2)$. Now, since  $f^*\tilde{\Theta}=4\Theta_H$, the restricted polarisation on $JH/G$ is not primitive and therefore it is of type $(2,2,4,\ldots,4)$.
If $G$ is non-isotropic, then one needs a polarisation of type $(4,\ldots,4)$ to get $\pi$ a polarised isogeny and hence $JH/G$ is $(1,4,\ldots,4)$ polarised. 

Now, the result follows from the fact that the complementary abelian subvariety has complementary type, see \cite[Cor. 12.1.5]{BL}.
\end{proof}

\begin{lem}
\label{kernel}
    The (restricted) kernel $\ker h_{\sigma}^*|_{P(h_{\sigma})}$ is trivial in the non-isotropic case, hence $P(k_{\sigma})$ is embedded in $J\tC$. The kernel has cardinality $2$ in the isotropic case and we have $P(k_{\sigma})^*\subset J\tC$. 
\end{lem}
\begin{proof}
    Since the coverings are \'etale, we have $\ker f^* = \langle \eta, \xi \rangle$, $\ker k_{\eta}^* = \langle \eta \rangle$ and $\ker h_{\eta}^* = \langle k_\eta^*(\xi) \rangle$. Since $$k_{\eta}^*JH \cap P(k_{\eta}) = k_{\eta}^*((\ker k_{\eta}^*)^{\bot})$$ it follows that in the non-isotropic case $k_{\eta}^*(\zeta) \notin P(h_{\eta})$ as $\xi \notin (\ker k_{\eta}^*)^{\bot}$.
    On the other hand, in the isotropic case  $\xi \in (\ker k_{\eta}^*)^{\bot}$, so $k_{\eta}^*(\xi) \in P(k_{\eta})$ and hence $h^*_{|P}(k_{\eta})$ is non-injective.
\end{proof}

In some cases, the Prym varieties are Jacobians themselves. We will use the following Lemma from \cite{BNOS} to recover some coverings from embeddings between Jacobians. 
\begin{lem}\label{key_lemma}
    Let $h:C'\to C$ be a covering that does not factorise via an \'etale covering $C'\to C''\to C$.
        Then $h^*: JC\to JC'$ is an embedding and any other embedding $i:JC\to JC'$ with $\Ima i=\Ima  h^*$ is (up to an isomorphism of $JC$) given by $h^*$.
\end{lem}
\begin{proof} 
   See \cite[Lemma 2.4]{BNOS}. 
\end{proof}
We also need a way to check whether the pullback map is an embedding.
\begin{cor}\label{coremb}
    Let $h:\tC\to H$ be a Klein covering. Then $h^*:JH\to J\tC$ is an embedding if and only if none of the arrows at the bottom of Diagram \ref{diag:Klein_cov_t} is \'etale. In particular, it is enough to compute the genera of the quotient curves and show that none of them equals $2g(H)-1$.
\end{cor}

\begin{defn}
By Strong Torelli Theorem (see \cite[Ex. 11.19]{BL}) for non-hyperelliptic curves $\tC$, we have $\Aut(J\tC)/(-1)=\Aut(\tC)$. Motivated by this result, we call an automorphism $\s\in \Aut(J\tC)$ {\it genuine} if it is induced by an automorphism of $\tC$ and {\it artificial} if it does not come from an automorphism on $\tC$. 
\end{defn}
\begin{rem}
    Note that if $\sigma\in \Aut(J\tC)$ is an involution, then $(1+\s)(1-\s)=0$, so $\Fix(\s)^0=\Ima(1+\s)=\ker(1-\s)^0$. This shows that the fixed locus of an artificial involution is the Prym variety of the covering given by a genuine involution.
\end{rem}

\subsection{Hyperelliptic curves}
Let $H$ be a smooth hyperelliptic genus $g$ curve. A Klein covering of $H$ is given by a 
Klein subgroup $\langle \eta, \xi \rangle \simeq \ZZ_2^2 \times \ZZ_2^2 $ of $JH[2]$. 
In the hyperelliptic case all the 2-torsion points 
can be written in terms of the set of Weierstrass points 
$W:=\{w_1, \ldots ,w_{2g+2}\}
\subset H $. We refer to \cite{DO} for the following description. 
 Given a subset $S \subset I=\{1, \ldots, 2g+2\}$ the divisor
\begin{equation} \label{W.pts}
\alpha_S= \sum_{i \in S} w_i - |S| \cdot w_{2g+2}  
\end{equation}
defines an element in $JH[2]$. Observe that $\alpha_S = \alpha_{I \setminus S}$.  
Denote by $E_g$  the $\FF_2$-vector space of functions $I \ra \FF_2$ having an even number of $0$'s and $1$'s modulo the constant functions $\{0,1\}$.  The elements of $E_g$ are represented by the subsets of even cardinality (up 
to complementary subset) and clearly $E_g \simeq \FF_2^{2g}$.  The correspondence $S \mapsto \alpha_S$ gives an isomorphism 
$E_g \simeq JC[2]$, so all the 2-torsion points on a hyperelliptic Jacobian are of the form \eqref{W.pts}. Moreover, $E_g$ carries a symmetric bilinear form 
$$
e: E_g \times E_g \ra \FF_2, \qquad e(S,T) = |S\cap T| \mbox{ mod } 2,
$$
which is a non-degenerated symplectic form. Under the above isomorphism, this form corresponds to the Weil pairing on the 2-torsion points on $JH$, for details see \cite[Section 5.2]{DO}.
A subspace $G \subset E_g$ is {\it isotropic} with respect to $e$ if $e(\alpha_S,\alpha_T) =  0$ (that is, if the divisors $\alpha_S $ and $\alpha_T$  share
an even number of points) for all $\alpha_S,  \alpha_T \in G $; otherwise the subspace $G$ is called {\it non-isotropic}.  

Since $2w_i$ is linearly equivalent to $2w_j$ for any $i,j$, we can also write any 2-torsion point on $JH[2]$ as a sum of differences of Weierstrass points. This presentation is better for the purpose of the paper as it does not involve choosing a base point $w_{2g+2}$.

\subsection{Four cases of Klein coverings}
From now on, we assume $H$ is a hyperelliptic curve of genus 3. We distinguish two non-isotropic cases of Klein coverings and two isotropic ones. For each of them we give an example of the generators of the Klein subgroup 
$\langle \eta, \xi \rangle$.
\bigskip

{\it Case I.1}. For $\eta= w_1-w_2,\ \xi= w_1-w_3$ the corresponding covering $\tilde C$ is hyperelliptic (\cite[Prop 4.4]{BO19}). The subgroup $\langle \eta, \xi \rangle$
is non-isotropic. The number of such covers for a fixed curve $H$ is the number of choices of triples among the 8 Weierstrass points, that is 56.  

{\it Case I.2}. For $\eta= w_1-w_2, \ \xi= w_1-w_3+w_4-w_5$, the subgroup $\langle \eta, \xi \rangle$ is non-isotropic. 

{\it Case II.1}. For $\eta= w_1-w_2+w_3-w_4, \ \xi= w_1-w_2+w_5-w_6$, the subgroup $\langle \eta, \xi \rangle$ is isotropic.

{\it Case II.2}. For $\eta= w_1-w_2, \ \xi= w_3-w_4$, the subgroup $\langle \eta, \xi \rangle$ is isotropic.
\bigskip

Let $\cRH_3^{V_4}$ be the moduli space parametrizing Klein coverings over hyperelliptic curves of genus 3. It is also the parameter space of triples:
$$
\cRH_3^{V_4}:= \{ [C, \langle \eta, \xi \rangle] \ : \ [C] \in \cM_3,\ \langle \eta, \xi \rangle\subset JC, \ \langle \eta, \xi \rangle \simeq \ZZ_2 \times \ZZ_2 \}.
$$

\begin{lem}\label{4components}
The moduli space $\cRH_3^{V_4}$ consists of four irreducible components, corresponding to the configuration of the Weierstrass points in the generators $ \eta, \ \xi$ as shown 
in I.1, I.2, II.1 and II.2. Moreover, the degree of the forgetful 
map  $\cRH_3^{V_4} \ra \cH_3$ on each component is 56, 280, 105, 210, respectively. 
\end{lem}
\begin{proof}
The irreducibility of components of $\cRH_3^{V_4}$ follows from
Propositions \ref{233}, \ref{2222}, \ref{224}
and \cite[Cor. 4.5]{BO24}, which give 
a description of the locus corresponding to each 
case I.1, I.2, II.1 and II.2. The  degrees correspond to the ways of choosing the generators of the Klein of group in terms of the 8 Weierstrass points.
Thus, these degrees are computed by 
$$
{8 \choose 3}= 56, \quad \frac{1}{2}{8 \choose 3}=280, \quad \frac{1}{6}{8 \choose 6}{6 \choose 2}{4 \choose 2}= 105,  \quad\frac{1}{2}{8 \choose 4}{4 \choose 2} = 210.
$$
Notice that the number of isotropic Klein subgroups is $\frac{63\cdot 30}{6}=315$ and the number of non-isotropic Klein subgroups is $\frac{63\cdot 33}{6}=336.$
\end{proof}

We will use the following crucial result on double coverings of hyperelliptic curves which rephrases a construction by Mumford. 
\begin{prop}\cite[p. 346]{M1} \label{Mumford}
Let $u_1,\ldots, u_{2k+2}, v_1, \ldots, v_{2g-2k}$ be a set of points in $\PP^1$.
Let $H_g$ be a genus $g$ hyperelliptic curve given by the double cover branched in all the points and let $C_k$, respectively $C_{g-k-1}$, be
the hyperelliptic curve of genus $k$, respectively $g-k-1$,  given by the double covering branched on the points $u_i$,  respectively $v_j$.
Then the normalisation of the fiber product $\tC$ of any two of these curves fits in the  commutative diagram: 
\begin{equation}\label{diag:Mum}
\xymatrix@R=.9cm@C=.6cm{
& \tC \ar[dr] \ar[dl] \ar[d] & \\
C_{k}\ar[dr] &  H_{g} \ar[d] & C_{g-k-1} \ar[dl]\\
&\PP^1& 
}
\end{equation}
The 2:1 map $\tC\to H_g$ is given by the line bundle $\cO_H(u_1-u_2+\ldots+u_{2k+1}-u_{2k+2})$. 
Moreover, a lift of the hyperelliptic involution from $H_g$ fixes preimages of $u_1,\ldots, u_{2k+2}$ whereas the other lift fixes preimages of $v_1,\ldots, v_{2g-2k}$.
\end{prop} 

In the case $k=0$, we have a criterion for the hyperellipticity of the top curve (see also \cite{BO24}).
\begin{cor}
    \label{farkas-lemma}
Let $H$ be a hyperelliptic curve of genus $g$ and $h:C\ra H$ an \'etale double covering defined by $\eta\in JH[2]\setminus \{0\}$.  Then $C$ is hyperelliptic if and only if 
$\eta= \cO_H(w_1-w_2)$, where $w_1, w_2 \in H $ are Weierstrass points.
\end{cor}



The following proposition is an extended version of an already known result, see \cite[Prop 2.1]{BNOS} and references therein.
\begin{prop}\label{autoofP}
Let $(P, \Xi)$ be an element in the image of the Prym map 
$\cP^{iso}_3$, respectively in $\cP^{ni}_3$ Case I.2.  Then it holds 
$$
B:=\{\psi \in \Aut (P,\Xi) \ : \ \psi^2=id, \ 
\psi_{|K(\Xi)}= \pm id \} = \langle \sigma, \tau, j, -1 \rangle \simeq \ZZ_2^4 
$$
where $j$ is any lift of the hyperelliptic involution from $H$.
\end{prop}
\begin{proof}
Firstly, note that $K(\Xi)=P\cap JH^*$ always contains $4$-torsion points. Since $\s$ and $\tau$ restrict to the identity on $JH^*$, they belong to $B$. Analogously, $j,-1\in B$ because they restrict to $(-1)$ on $JH^*$. Note that $j\neq -1$ because $\tC$ is not hyperelliptic.

    Let $\psi \in B$ with restriction to $K(\Xi)$ equals $\pm id$. Then, there is an automorphism $\Tilde{\psi}: J\tC \to J\tC$ such that the following diagram commutes:


\begin{equation}
        \begin{tikzcd}
0 \arrow[r] & K(\Xi) \arrow[d, equal] \arrow[r] & f^*JH \times P \arrow[d, "{(\pm\id,\psi)}"] \arrow[r, "\mu"] & J\tC \arrow[d, "\Tilde{\psi}"] \arrow[r] & 0 \\
0 \arrow[r] & K(\Xi) \arrow[r]                                & f^*JH \times P \arrow[r, "\mu"]                           & J\tC \arrow[r]                           & 0
\end{tikzcd}
\end{equation}


where $\mu$ is the addition map. Repeating the argument from \cite{BS25}, we get that $\tilde{\psi}$ is a polarised isomorphism. Moreover, by the Strong Torelli Theorem  we have that either $\tilde{\psi}$ or $-\tilde{\psi}$ comes from an automorphism of $\tC$. Note that $-\tilde{\psi}=\widetilde{(-\psi)}$. However, by construction of $\tilde{\psi}$, we have that it restricts to $\pm 1$ on $f^*JH$ which means that (if it comes from an automorphism of $\tC$) it is a lift either of identity on $H$ or hyperelliptic involution on $H$. 
\end{proof}

\begin{cor}\label{Bplus}
    Under the same assumptions as in Proposition \ref{autoofP}, we get that 
    $$
B^+:=\{\psi \in \Aut (P,\Xi) \ : \ \psi^2=id, \ 
\psi_{|K(\Xi)}= id \} = \langle \sigma, \tau, -j \rangle \simeq \ZZ_2^3. 
$$
\end{cor}

In order to prove irreducibility of considered spaces of Klein coverings, we define the following moduli space.
\begin{defn}
Assume $k_1, k_2,\ldots, k_n$ is a set of non-negative integers with $k_1+\ldots+k_n>3$.
Let $\PP_{k_1,\ldots,k_n}$ be the moduli space of $k_1+\ldots+k_n$ marked points in $\PP^1$  up to an isomorphism that preserves subsets of cardinality $k_1,\ldots, k_n$. 
\end{defn}
\begin{rem}
    This space can be seen as a quotient of the space of $k_n$ marked points of $\PP^1$ (if $k_1=\ldots=k_n=1$ the spaces are equal), hence it is an irreducible variety of dimension $k_1+\ldots+k_n-3$.
\end{rem}
\begin{defn}
    If $k_{i}=\ldots=k_j$, for $i\neq j$ we denote by $\PP'_{k_1,\ldots,k_n}$ the quotient space where we allow to permute the sets with equal number of elements. In particular $\PP'_{1,\ldots,1}$ is the projective space with an unordered $n$-tuple of points
\end{defn}

\bigskip
    
\section{Non-isotropic Klein coverings of type I.2}
Recall that in this case, there are Weierstrass points such that $\eta= w_1-w_2, \ \xi= w_1-w_3+w_4-w_5$ and the subgroup $\langle \eta, \xi \rangle$
is non-isotropic. 

We start by showing the following characterisation of the moduli of genus 3 non-isotropic Klein coverings of type I.2.
\begin{prop}\label{233}
The following data are equivalent.
    \begin{enumerate}
        \item the space $\PP'_{2,3,3}$;
        \item the space of (non-isotropic) Klein coverings of hyperelliptic genus 3 curves of type $I.2$.
    \end{enumerate}
\end{prop}
\begin{proof}
    Starting from an element $(\{u_1,u_2\},\{u_3,u_4,u_5\}.\{u_6, u_7, u_8\})$, we can construct a genus 3 hyperelliptic curve as a double covering of $\PP^1$ branched in these 8 points. Then preimages of these points become Weierstrass points, denoted by $w_i$. Now we choose the following Klein group, take a point from the pair and all points from a triple or both points from the pair. In coordinates, we get $\{w_1-w_2, w_1-w_3+w_4-w_5, w_2-w_3+w_4-w_5\}$. Note that the group does not depend on the labelling of Weierstrass points and by properties of Weirerstrass points, the group does not depend on whether we choose one or the other triple. 

    A Klein covering of type $I.2$ of hyperelliptic genus 3 curve, is constructed uniquely from the curve and it Klein subgroup of 2-torsion points of the form $\{w_1-w_2, w_1-w_3+w_4-w_5, w_2-w_3+w_4-w_5\}$. As we already noted, there is an ambiguity of points that are written as differences of 4 Weierstrass points, so the covering given by $\{w_1-w_2, w_1-w_6+w_7-w_8, w_2-w_6+w_7-w_8\}$ is isomorphic to the one we started with.
    Hence, looking at their images under hyperelliptic covering to $\PP^1$ we get an element of $\PP'_{2,3,3}$.
\end{proof}

Let $(H, \langle \eta, \xi \rangle)$ be a Klein covering of type $I.2$ and  $\langle \sigma, \tau, j\rangle \simeq$ a group of automorphisms acting on $\tC$ where $j$ is a lift of the hyperelliptic involution. 
\begin{prop}
There is a choice of the lift $j$ such that the number of fixed points in $\tC$ given by each involutions are
\begin{eqnarray*}
|\Fix(j)|=|\Fix(j\s)|=12, \quad |\Fix(j\tau)|=|\Fix(j\sigma\tau)|=4, \\\quad |\Fix(\sigma)|=|\Fix(\tau)|= |\Fix(\sigma\tau)|=0.
\end{eqnarray*}
\end{prop}
\begin{proof}
Recall that the total number of fixed points on $\tC$ of lifts of the hyperelliptic involution equals 32.
In order to find the number of fixed points, note that $C\eta$ is hyperelliptic, hence we can choose a lift of hyperelliptic involution to $C_\eta$ in a way it is hyperelliptic (with 12 fixed points). Moreover a map $\tC\to C_\eta$ is given by the pullback of $w_1-w_3+w_4-w_5$ that can be written as difference of 6 Weierstrass points on $C_\eta$. Hence, by Proposition \ref{Mumford}, the lifts, denoted by $j$ and $j\s$ both have 12 fixed points.
Now, looking at $\tC\to C_\xi\to H$ we get that $|\Fix(j)|+|\Fix(j\tau)|=16$, hence $|\Fix(j\tau)|=4$ and finally,  $|\Fix(j\sigma\tau)|=4$.
\end{proof}
\begin{cor}
The genera of the quotient curves $C_{\alpha}=\tC/\langle \alpha\rangle $ for $\alpha \in \Aut(\tC)$ are
\begin{eqnarray*}
g(C_j)=g(C_{j\s})=2, \quad  g(C_{j\tau})= g(C_{j\sigma\tau})=4,\\
 g(\tC/\langle \s,j\tau \rangle )  =2, \quad  
g(\tC/\langle j, \tau \rangle ) = 1
\end{eqnarray*}
\end{cor}
In the commutative Diagram \eqref{diag:tower_curves_I.2} the genera
of the curves from the top are 
9,5,4,2 and 1. 
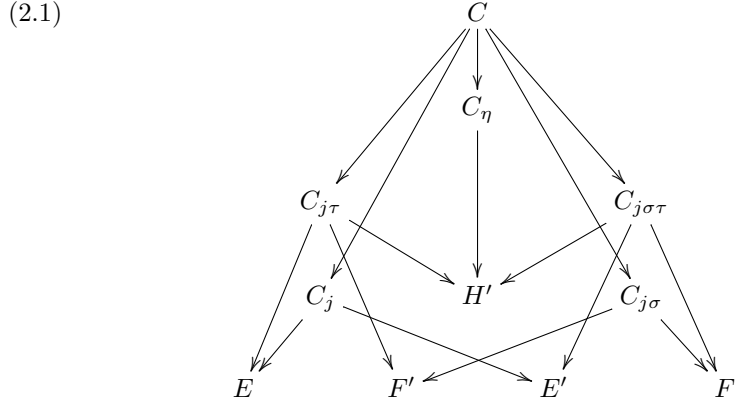
\begin{equation}\label{diag:tower_curves_I.2}
\xymatrix@R=.7cm@C=.4cm{
&&& \tC \ar[d] \ar[ddll] \ar[ddrr] \ar[dddll] \ar[dddrr]&&& \\
&&& C_{\eta}\ar[dd] &&& \\
& C_{j\tau} \ar[ddl] \ar[ddr] \ar[drr]&&&& C_{j\sigma\tau} \ar[ddl] \ar[ddr]\ar[dll] &\\
&C_j \ar[drrr]\ar[dl]&&  H' && C_{j\s} \ar[dr] \ar[dlll] &\\
E && F' && E' && F
}
\end{equation}


\begin{prop} \label{isotyp_I.2PB}
The isotypical decomposition of $J\tC$ with respect to $\left<\s,\tau,j\right>$
is given by 
$$
  J\tC  =   JH^* \boxplus JH' \boxplus E \boxplus F \boxplus E' \boxplus F'
$$ 
where  $H':= \tC/\langle \s,j\tau \rangle$, $E: = \tC/ \langle j, \tau \rangle $, 
$E':= \tC/\langle j, \sigma\tau \rangle$,
$F:= \tC/\langle j\s, \tau \rangle$ and
$F':= \tC/\langle j\s, \sigma\tau \rangle$. In particular, the isotypical decomposition of $P$ with respect to $B$ is 
$$
 P= JH' \boxplus E \boxplus F \boxplus E' \boxplus F'.
$$ 

Moreover, we have the isogenies 
$$
JC_j \sim E \times E', \quad JC_{j\s} \sim F \times F', \quad JC_{j\tau} \sim JH' \times  E\times F',
\quad JC_{j\sigma\tau} \sim JH' \times  E'\times F.
$$
\end{prop}
\begin{proof}
 The first part consists of checking all representations of $\ZZ_2^3$. In particular, it is enough to use decomposition from Lemma \ref{basicisodec} and check how $j$ acts on components.
 For example taking a representation $r$ for which $r(\s)=1, \ r(\tau)=1, \ r(j)=-1$ gives $JH^*$ (as $j$ is hyperelliptic on $H$). Analogously, $r$ acting as $r(\s)=1, \  r(\tau)=-1, \  r(j)=-1$ gives $JH'$ whereas if 
 $r$ acts as $r(\s)=-1, \ r(\tau)=1, \ r(j)=1$ one obtains a complementary abelian subvariety to $JH'$ in $P(C_\tau/H)$ that is $0$, so it does not appear in the decomposition. By looking at the genera of the quotient curves of $\tC$ by an
 involution, one sees that all the  components of the isotypical decomposition actually embed into $J\tC$ (see Corolloary \ref{coremb}). The isogenies follows from checking which components are quotients. For example 
 $E=C_j/\langle \s \rangle $, $E'=C_j/\langle\s\tau \rangle$, so $JC_j\cong E\times E'$. Note here, that $C_j/\langle\s \rangle=C_j/\langle j\s\rangle=\PP^1$ because $j\s$ has at least 6 fixed points on $C_j$ hence it is a hyperelliptic involution on $C_j$.
 \end{proof}

Recall from Proposition \ref{autoofP}, that  $B\simeq \ZZ_2^4 $ is determined by the polarised 
automorphisms of $P$ that  fix (or act as -1) on the kernel of the polarisation $\Xi$. The group $B$ contains 8 artificial and 8 genuine involutions. Observe that the artificial 
involution $-1$ is  distinguished  since it is the only 
involution having a zero-dimensional fixed locus. 

Note that $JH' \subset J\tC$ with restricted polarisation of type $(4,4)$, by means of Torelli theorem one can recover the genus 2 curve $H'$ from the isotypical decomposition of $P$. On the other hand, the Jacobians $JC_j$
and $JC_{j\s}$ are embedded in $J\tC$ with polarisation type $(2,2)$. According to Poincaré Reducibility Theorem  there exists an isogeny 
$$
\psi: JC_j \times (JC_j)^c \lra P
$$
as polarised abelian varieties, where  $(JC_j)^c$ denotes the complementary subvariety to $(JC_j)$ inside 
$P$. Since $\Ker(\psi^*(\Xi)) \supset \ZZ_4^2$ the polarisation type of  $(JC_j)^c$ is of the form $(*,*,*, 4a)$ with $a\in \ZZ_+$. 
Notice that $JC_j = \Fix(j)^0$ and $j$
acts on $(JC_j)^c$ as $-1$, equivalently $(JC_j)^c = \Fix(-j)^0 $ and
$-j$ is an artificial involution.
Similarly, the Jacobians $JC_{j\tau}$  and $JC_{j\sigma\tau}$ 
have polarisation type $(2,2,2,2)$. The above argument shows that $(JC_{j\tau})^c$ and $(JC_{j\sigma\tau})^c$  have polarisation type of the form $(*, 4b)$ with $b\in \ZZ_+$, 
and  are contained in the fix loci of artificial involutions ($-j\tau$ and $-j\s\tau$) of  $J\tC$. On the other hand,  for  $\alpha \in \{ \sigma,\tau, \sigma\tau \}\subset B$, $\alpha $ is genuine involution and 
$\Fix(\alpha) =  JC_{\alpha}^*
\cap P=P(C_\alpha/H)$ is an abelian subvariety with restricted polarisation of type 
$(4,4)$. Moreover, $\Fix(-\alpha)^0=P(\tC/C_\alpha)\subseteq P$ is of type $(2,2,2,2)$.
In this way, we can compute the restricted polarisation of all subvarieties of $P$ of dimensions 1, 2 and 4.


\begin{lem}\label{fourfolds}
    Using the notation from Diagram \ref{diag:tower_curves_I.2}, the abelian fourfolds $$\Fix(-j),\ \Fix(-j\s),\ \Fix(-\tau),\ \Fix(-\s\tau)$$ are not Jacobians of genus 4 curves with restricted polarisation of type $(2,2,2,2)$.
\end{lem}
\begin{proof}
By what we have written above the lemma, we get that $\Fix(-j)$ and $\Fix(-j\s)$ do not have a $(2,2,2,2)$ polarisation type.
As for the other two, observe that 
$C_{\tau}$ and $C_{\s\tau}$ are bielliptic curve since the quotient by $j$ of any of them
is an elliptic curve. It is known that bielliptic curves of genus 5 are not trigonal (\cite{T}). Then, according to \cite[Main theorem]{Sho}, the Prym variety of a bielliptic genus 5 curve is not a Jacobian. This shows that $\Fix(-\tau),\ \Fix(-\s\tau)$ are not Jacobians.
\end{proof}

\begin{thm}\label{thmI.2}
The  Prym map 
$\cP^{ni}_3 :  \cRH_3^{ni} \ra  \cA_6^{(1,1,1,1,4,4)}$ is injective on the irreducible component consisting of Klein coverings of type I.2.

\end{thm}
\begin{proof}

Let $(P, \Xi)$ be  the image of a Klein covering of type $I.2$ under $\cP^{ni}_3$. First note that, by Proposition \ref{autoofP}, $(P, \Xi)$ determines the group of involutions $B\simeq \ZZ^4_2$  on $P$ and by Corollary \ref{Bplus} we get the group $B^+$. One checks the dimensions of the fixed loci of involutions of $B^+$ (see Diagram \ref{diag:tower_curves_I.2}) to distinguish 
a unique involution, called $\s$ whose fixed locus is the only two dimensional component in the isotypical decomposition of $P$ with respect to $B$, see Proposition \ref{isotyp_I.2PB}. Denote by $JH'=\Fix(\s)$. 
Now, there are precisely 8 involutions that act on $JH'$ as $(-1)$, (namely $-1,\ -\s,\ \tau,\ \s\tau,\ j,\ j\s,\ -j\tau,\ -j\s\tau$). Only two of them have a fixed locus being a surface with restricted polarisation type $(2,2)$. We call them $j$ and $j\s$. 

Now, consider the subvarieties of dimension 4 containing $JH'$. They are 
\begin{eqnarray*}
JH'+E+F,\ JH'+E'+F, \ JH'+E+F',\\
JH'+E'+F', \ JH'+E+E',\  JH'+F+F',
\end{eqnarray*}
where $E,E',F,F'$ are elliptic curves that come from isotypical decomposition with respect to $B$.
By Lemma \ref{fourfolds} only two of these fourfolds are Jacobians of genus 4 curves with restricted polarisation of type $(2,2,2,2)$, called $JC_{j\tau}$ and $JC_{j\s\tau}$ and they are fixed locus of involutions called $j\tau$ and $j\s\tau$ respectively.

Observe that by construction $JH' \subset JC_{j\tau} $ and $ JH' \subset JC_{j\sigma\tau} $. According to Lemma \ref{key_lemma}, these inclusions determine double coverings $h_1:  C_{j\tau}\ra H'$ and 
$h_2:  C_{j\sigma\tau}\ra H' $, each of them ramified in two points, given as the quotient by the  involution $\s$. 
Then the (normalised) fibered product $\tC:=\widetilde{C_{j\sigma\tau} \times_{H'}   C_{j\tau}} $ is a smooth curve of genus 9, see Diagram \ref{fibred_product_I.2}. 
\begin{equation}\label{fibred_product_I.2}
\xymatrix@R=1.2cm@C=1.4cm{
 \widetilde{C_{j\sigma\tau} \times_{H'}   C_{j\tau}} \ar[d]_{2:1}^{+4} \ar[r]^{ \quad 2:1}_{\quad +4} &   C_{j\sigma\tau}  \ar[d]^{2:1}_{+2}  \\
 C_{j\tau} \ar[r]_{2:1}^{+2} &  H' 
}
\end{equation}

Note that all maps are double coverings, hence given by involutions, so they form a subdiagram of a (branched) Klein covering. This corresponds to  the following subdiagram 
\begin{equation}\label{diag:MumI2}
\xymatrix@R=.9cm@C=.6cm{
& \tC \ar[dr] \ar[dl] \ar[d] & \\
C_{j\tau}\ar[dr] &  C_{\s} \ar[d] & C_{j\s\tau} \ar[dl]\\
&H'& 
}
\end{equation}
of Diagram \ref{diag:tower_curves_I.2}.
This shows that $\tC$ satisfies the universal property of the fibered product and therefore is isomorphic to 
$\widetilde{C_{j\sigma\tau} \times_{H'}   C_{j\tau}}$.
Since on $\tC$ acts  $\langle \s,j\tau \rangle \simeq \ZZ_2^2$, we recover the full set of involutions by lifting the hyperelliptic involution from $H'$ to $\tC$ (by construction it is liftable) and therefore the map $\tC\to H\in\cRH_3^{ni}$.

 \end{proof}

\begin{rem}
Following Diagram \ref{diag:tower_curves_I.2} and Proposition \ref{isotyp_I.2PB} consider the covering $\tC\to {C_\tau}$. Since $JC_{\tau}=JH^*\boxplus E\boxplus F$, we get that the Prym variety of the covering equals $P(\tC/C_\tau)=JH'\boxplus E'\boxplus F'$. 
The involution $j\tau$ is genuine and its restriction to $P$ has a fixed point set equal to  $JC_{j\tau}\cap P(\tC/C_{\tau})=JH'\boxplus E'$, hence is of codimension 1 in $P(\tC/C_{\tau})$. This shows that $j\tau$ is a pseudoreflection of geometric origin according to \cite{ALN}. 
Moreover, similar computation applies to the involution $j\s\tau$, so it is also a pseudoreflection in $P(\tC/C_{\tau})$. Therefore, we have found an example of a Prym variety with two pseudoreflections. 
Since the dimension of the Prym variety equals $4$, the family is different from the family of intermediate Jacobians of cubics with two Eckhardt points presented in \cite{ALN}.
    \end{rem}

\section{Isotropic Klein coverings of type II.1}

In this section we consider the case of the isotropic Klein coverings $f:\tC \ra H $ with associated isotropic subgroup generated by elements of the form  $\eta= w_1-w_2+w_3-w_4, \ \xi=w_1-w_2+w_5-w_6$. 

\begin{prop}\label{2222}
The following data are equivalent.
    \begin{enumerate}
        \item the space $\PP'_{2,2,2,2}$;
        \item the space of (isotropic) Klein coverings of hyperelliptic genus 3 curves of type $II.1$.
    \end{enumerate}
\end{prop}
\begin{proof}
    Starting from an element $(\{u_1,u_2\},\{u_3,u_4\},\{u_5,u_6\}.\{u_7, u_8\})
    \in \PP'_{2,2,2,2}$, we consider the unique genus 3 hyperelliptic curve which is double covering of $\PP^1$ branched in these 8 points. Then the preimages of these points become Weierstrass points, denoted by $w_i$,
     $i=1, \ldots, 8$. Choose the unique Klein group generated by the differences of the pairs, i.e., take any 3 pairs of points that are not pairwise complementary to each other. We get for example $\{w_1+w_2-w_3-w_4, w_1+w_2-w_5-w_6, w_1+w_2-w_7-w_8\}$. Note that the group does not depend on the labeling of Weierstrass points and by the properties of the Weierstrass points the group (as a subgroup of the Jacobian) is unique. 

    A Klein covering of type $I.2$ of hyperelliptic genus 3 curve, is constructed uniquely from the curve and its Klein subgroup of 2-torsion points of the form $\{w_1+w_2-w_3-w_4, w_1+w_2-w_5-w_6, w_1+w_2-w_7-w_8\}$. 
    Hence, looking at the images of the Weierstrass under the hyperelliptic covering in $\PP^1$ we obtain an element of $\PP'_{2,2,2,2}$.
\end{proof}

Let $j$ be a lift of the hyperelliptic involution on $\tC$. The possible lifts are indistinguishable
as the following proposition shows.
\begin{prop}\label{lift1}
The cardinality of the fixed points of the involutions on $\tC$ are 
$$
|\Fix ( j)|= |\Fix ( j\tau)|= |\Fix ( j\sigma)| = |\Fix ( j\tau \sigma)| =8.
$$
\end{prop}
\begin{proof}
    Let $\iota$ be the hyperelliptic involution on $H$. By Proposition \ref{Mumford}, we have that both its lifts to $C_\eta$ have 8 fixed points each. Hence the sum of the fixed points of the lifts are 16 for a pair $j,j\s$, as well as for a pair $j\tau,j\s\tau$. We proceed similarly for $C_\xi$ and $C_{\eta+\xi}$ and obtain that the cardinalities of the fixed point set of any pair of lifts add up to 16. This means that all lifts have 8 fixed points. 
\end{proof}

\begin{cor}\label{generaII.1}
The genus of the quotient curve $C_{\alpha}$ 
of $\tC$ by an  involution $\alpha \in \{j,j\sigma, j\tau,  j\sigma\tau \} \subset 
\Aut(\tC)$ is 3. 
\end{cor}
In the following tower of quotient curves the genera at each row from the top down are 9,5,3 and 1.
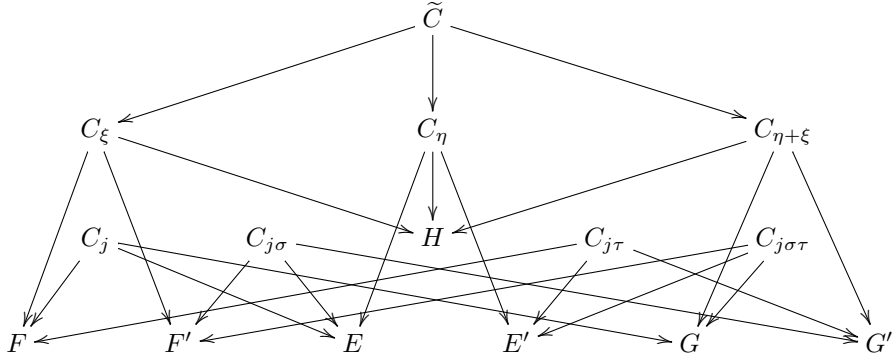
\begin{equation}\label{diag:tower_curves_II.1}
\xymatrix@R=.9cm@C=.5cm{
&&&&& \tC \ar[d] \ar[dllll] \ar[drrrr] &&&&& \\
&C_{\xi} \ar[ddl] \ar[ddr] \ar[drrrr]&&&& C_{\eta}\ar[d] 
\ar[ddl]  \ar[ddr] &&&& C_{\eta+\xi} \ar[dllll] \ar[ddl]  \ar[ddr] & \\
&C_j \ar[dl] \ar[drrr] \ar[drrrrrrr] &&  C_{j\sigma} \ar[dr] \ar[dl]\ar[drrrrrrr]
&& H && C_{j\tau} \ar[drrr] \ar[dl]  \ar[dlllllll] &&  C_{j\sigma\tau}\ar[dl] \ar[dlll]  \ar[dlllllll] &\\
F && F' && E &&  E' && G && G'
}
\end{equation}

\begin{prop} \label{isotyp_II.1}
The isotypical decomposition of $J\tC$ with respect to $\left<\s,\tau,j\right>$
is given by 
\begin{eqnarray*}
J\tC  
  &=&  JH^* \boxplus E \boxplus E'\boxplus F \boxplus F'\boxplus G \boxplus G'.
\end{eqnarray*}
where 
\begin{eqnarray*}
E=\tC/\langle j, j\sigma \rangle, \quad 
F=\tC/\langle j, j\tau \rangle, \quad 
G=\tC/\langle j, j\sigma\tau \rangle, \\
E'=\tC/\langle j\tau, j\sigma\tau \rangle, \quad 
F'=\tC/\langle j\sigma, j\tau\sigma \rangle, \quad 
G'=\tC/\langle j\sigma, j\tau \rangle
\end{eqnarray*}
The restricted polarisations to the isotypical components are of type $(2,2,4)$ on the image of the Jacobian $JH$ and $(4)$ on the elliptic curves $E, F, G, E',F', G'$. Moreover, we have the isogenies
$$
JC_j \sim E\times F\times G, \ 
\ JC_{j\sigma} \sim E \times F' \times G',
\  JC_{j\tau} \sim E' \times F \times G',
\  JC_{j\sigma\tau} \sim E' \times F' \times G,
$$
\end{prop}
\begin{proof}
Note that $\ZZ_2^3$ has 7 non-trivial representations and one can check that the decomposition fill in all of them. Moreover, by Diagram \ref{diag:tower_curves_II.1} one checks that for the curve $E$, we have three quotient maps from $C_\eta, C_{j\s}, C_j$ that are curves of genus $g>1$, so a quotient map $\tC\to E$ does not factorize via a cyclic \'etale covering, hence $E$ is embedded into $J\tC$, see Corollary \ref{coremb}. Similarly for the other elliptic curves. The type of the restricted polarisation on $JH^*$ has been established in Lemma \ref{polonJH} and the exponent of the elliptic curves equals the degree of the quotient map that is $4$. The isogenies follow  from the fact that the genera of the curves equal 3 and there are 3 quotient maps for each curve, see Diagram \ref{diag:tower_curves_II.1}.
\end{proof}
\begin{rem}
    One can actually check that $JC_j=E\boxplus F\boxplus G$ as an abelian subvariety of $P(\tC/H)$ or $J\tC$.  
\end{rem}

\begin{lem} \label{intersection_Jacobians_II.1}
The intersection of the Jacobians $JC_{j} \cap JC_{j\sigma}$ embedded in $J\tC$ is the elliptic curve $E$.
\end{lem}
\begin{proof}
    Note that $JC_j=\Fix(j), JC_{j\s}=\Fix(j\s)$ and $E=\Fix(\langle j,j\s\rangle)$ is the intersection of the two.
\end{proof}

As before, we would like to distinguish Jacobians inside the Prym variety. In order to do this, note that we have 8 genuine involutions and 8 artificial ones.
The dimensions of the fixed loci of these
involutions are as follows. For $\pm1$, the fixed locus is $J\tC$ or $0$. For $\s,\tau,\s\tau$ the dimension equals 2, hence for artificial ones $-\s,-\tau,-\s\tau$ we get that the dimension equals 4. 
As stated before, for $j,j\s,j\tau,j\s\tau$, the dimension of the fixed locus equals 3. Moreover, we can proceed as in case $I.2$ to distinguish genuine from artificial involutions in the following way. Since $j$ is an involution with fixed points, the restricted polarisation to $JC_j$ is of type $(2,2,2)$. Since $\Ker(\psi^*(\Xi)) \supset \ZZ_4^2$ the polarisation type of  $(JC_j)^c=\Fix(-j)$ is of the form $(*,*, 4a)$ with $a\in \ZZ_+$. Similarly for the other involutions. As a consequence, one can distinguish genuine from artificial involutions in this case.

\begin{thm}
The  Prym map 
$\cP^{iso}_3 :  \cRH_3^{iso} \ra  \cA_6^{(1,1,1,2,2,4)}$  is injective on the irreducible component consisting of Klein coverings of type II.1.
\end{thm}
\begin{proof}
Let $(P, \Xi)$ be  the image of a Klein covering of type $II.1$ under $\cP^{iso}_3$. First note that, by Proposition \ref{autoofP}, $(P, \Xi)$ determines the group of involutions $B\simeq \ZZ^4_2$  on $P$.
One then look at the fixed loci of the involutions together with the restricted polarisation types to distinguish $j,j\s,j\tau,j\s\tau$ and $JC_j,JC_{j\s},JC_{j\tau},JC_{j\s\tau}$ respectively. Now, one chooses any two of these, say $j$ and $j\s$ and by Torelli Theorem one recovers $C_j$ and $C_{j\s}$.
By Lemma \ref{intersection_Jacobians_II.1} and Lemma \ref{key_lemma} we obtain the maps $C_j\to E$ and $C_{j\s}\to E$, so we  have the following fibered product.

\begin{equation}\label{fibred_product_II.1}
\xymatrix@R=1.2cm@C=1.4cm{
 \widetilde{C_{j} \times_{E}   C_{j\sigma}} \ar[d]
 \ar[r]
 &   C_{j}  \ar[d]_{2:1}^{+4}  \\
 C_{j\sigma} \ar[r]^{2:1}_{+4} &  E 
}
\end{equation}
where $\widetilde{\ \ \ }$ is the  normalisation of the (possibly) singular curve $C_{j} \times_{E}   C_{j\sigma}.$

Note that all maps are double coverings, hence given by involutions, so they form a subdiagram of a (branched) Klein covering. Looking at Diagram \ref{diag:tower_curves_II.1}, we recover this diagram as the following:
\begin{equation}\label{diag:MumII1}
\xymatrix@R=.9cm@C=.6cm{
& \tC \ar[dr] \ar[dl] \ar[d] & \\
C_{j\s}\ar[dr] &  C_{\eta} \ar[d] & C_{j} \ar[dl]\\
&E& 
}
\end{equation}
This shows that $\tC$ satisfies the universal property and hence it is isomorphic to $\widetilde{C_{j} \times_{E}   C_{j\sigma}}$. Moreover, by construction, from the fibered product, we recover the map $\tC\to C_\eta$ (as given by a composition of involutions).

We perform the same procedure for another pair of involutions on $P$, for example $j,j\tau$. By the uniqueness of the fibered product, we have to obtain $\tC$ again, now with an involution $\tau$. In this way, we recover the Klein covering $\tC\to H$, hence proving that the map is injective. 
\end{proof}

\section{Isotropic Klein coverings of type II.2}
Let $f:\tC \ra H $ be the Klein covering induced by the 
isotropic subgroup generated by $\eta= w_1-w_2, \ \xi= w_3-w_4$. According to Corollary \ref{farkas-lemma}, the curves $C_{\eta}$ and $C_{\xi}$ are  hyperelliptic but $C_{\eta +\xi}$ is not.

\begin{prop}\label{224}
The following data are equivalent:
    \begin{enumerate}
        \item the space $\PP'_{2,2,4}$;
        \item the space of (isotropic) Klein coverings of hyperelliptic genus 3 curves of type $II.2$.
    \end{enumerate}
\end{prop}
\begin{proof}
    Starting from an element $(\{u_1,u_2\},\{u_3,u_4\}.\{u_5,u_6, u_7, u_8\})\in \PP'_{2,2,4}$, we can construct a genus 3 hyperelliptic curve as a double covering of $\PP^1$ branched in these 8 points. 
    Then preimages of these points become Weierstrass points, denoted by $w_i$. 
    Now, we choose the Klein group by taking the first pair, the second pair and then both pairs. In coordinates, we get $\{w_1-w_2, w_3-w_4, w_1-w_2+w_3-w_4\}$.
    Certainly, if we swap the pairs, we will obtain the same Klein group, hence the construction from $\PP'_{2,2,4}$ is well defined. Note that the group does not 
    depend on the labeling of the Weierstrass points and by a property of the Weierstrass points, one can choose the last quadruple as the sum of two generators of the group. 

    A Klein covering of type $II.2$ of a hyperelliptic genus 3 curve is constructed uniquely from the curve and from its Klein subgroup of 2-torsion points of the form $\{w_1-w_2, w_3-w_4, w_1-w_2+w_3-w_4\}$. Hence, by considering  their images under hyperelliptic covering to $\PP^1$ we obtain an element of $\PP'_{2,2,4}$.
\end{proof}

As before, we  describe the lifts of the hyperelliptic involution in the following proposition.
\begin{prop}\label{liftII2}
There is a unique lift, denoted by $j$, on $\tC$ with exactly 16 fixed points. Moreover,  $|\Fix ( j\tau \sigma)|= 0 $ and    $|\Fix ( j\tau)|= |\Fix ( j\sigma)| =8 $.
\end{prop}
\begin{proof}
Using Proposition \ref{Mumford}, we see that the hyperelliptic involution lifts to a hyperelliptic involution on $C_\eta$ (with 12 fixed points) and another involution with 4 fixed points. Moreover, the pullback of $\xi$ on $C_\eta$ can be written using 4 Weierstrass points (which are the preimages of $w_3,w_4$). 
Hence, using again Proposition \ref{Mumford} the lifts of the hyperelliptic involution to $\tC$ have $16$ and $8$ fixed points. 
We denote them by $j$ and $j\s$. We can perform the same procedure for $C_{\xi}$ to get other involutions, called $j'$ and $j'\tau$. 
Since the total number of fixed points of the lifts of the hyperelliptic involution from $H$ equals 32, we have  that $j=j'$ and hence 
the remaining involution $j\s\tau$ is fixed point free. 
\end{proof}

As a consequence of Proposition \ref{liftII2} we have:  
\begin{cor}\label{genera1}
The genera of the quotient curves of $\tC$ by an  involution $\alpha \in \Aut(\tC)$ are $g(C_j)=1, \ g(C_{j\sigma})= g(C_{j\tau})= 3$ and $g(C_{j\tau\sigma})=5 $.
\end{cor}
The non hyperelliptic curve $C_{\eta+\xi}$ admits two lifts of the hyperelliptic involution on $H$, denoted (with an abuse of notation) by 
$j$ and $j':=j\sigma=j\tau$, each of them with 8 fixed points. Let $E:=C_{\s\tau}/j$ and $F= C_{\s\tau}/j'$ be the quotient elliptic curves and let $E_j=\tC/j$. We have the following commutative diagram, where the genera of the curves 
on each row (from the top) are 9,5,3,2,1,0. 

\begin{equation}\label{diag:partial_tower_II2}
\xymatrix@R=.6cm@C=.9cm{
 & & & \tC \ar[dlll] \ar@/_1.0pc/[ddll] \ar[dr] \ar[dl]\ar[dd] \ar[dr] \ar@/^2.0pc/[ddddrrr] \ar[drrr] & & &\\
C_{\tau} \ar[dd] &&  C_{\sigma}\ar[dd] && C_{\sigma\tau} \ar[d] \ar[dddr] \ar[dddl]&&  C_{j\sigma\tau} \\
 & C_{j\sigma}  \ar[dl] \ar@/_1.5pc/[ddrr] && C_{j\tau} \ar[dl] \ar[dd]& H \ar[ddd]&& \\
D_{\tau} && D_{\sigma} &&&&  \\
 &&& F \ar[dr] && E \ar[dl] & E_j  \ar[l]^{2:1} \\
&&&& \PP^1 && 
}
\end{equation}

\begin{prop} \label{isotyp_II.2}
The isotypical decomposition of $J\tC$ with respect to $\left<\s,\tau,j\right>$
is given by 
$$
  J\tC  =  JH^* \boxplus JD_\s^* \boxplus JD_{\tau}^* \boxplus  F\boxplus E_j
$$ 
where $D_\s=\tC/\left<\s,j\tau\right>, D_\tau=\tC/\left<\tau,j\s\right>$. 
Note that the restricted polarisation on $E_j=\tC/j$ is $(2)$, whereas on $F=\tC/\left<\s\tau,j\s\right>$ it is $(4)$.
\end{prop}

\begin{proof}
As before, note that $JD_\s=P(C_\s/H), \ JD_{\tau}=P(C_\tau/H)$. To see the last part of the decomposition, note that $P(C_{\s\tau}/H)=E\boxplus F$, however, both $P(C_{\s\tau}/H)$ and $E$ are not embedded in $J\tC$. Since $g(C_{\s\tau})>1,\ g(C_{j\s})>1,\ g(C_{j\tau})>1$, we have that $F$ is embedded in $J\tC$. And so $P(C_{\s\tau}/H)^*=E_j\boxplus F$.
\end{proof}

\begin{lem} \label{intersection_Jacobians}
The curve $C_{j\s}$ 
(respectively $C_{j\tau}$) is a double covering of $F$ and a double covering of $D_\tau$ 
(respectively $D_{\s}$). In particular,  both curves are hyperelliptic.
The intersection of the Jacobians $JC_{j\s} \cap JC_{j\tau}$ embedded in $J\tC$ is the elliptic curve $F$.
\end{lem}

\begin{proof}
Firstly, note that the quotient of  $C_{j\s}$ 
(respectively $C_{j\tau}$) by $j\tau$ (respectively $j\s$) is $F$, whereas the quotient  by $j\tau$ is the curve  $D_\tau$ 
(respectively $D_\s$).
Secondly, notice that the Jacobians of $C_{j\s}$ and $C_{j\tau}$ can be described as the image of certain endomorphisms of $J\tC$:
$$
JC_{j\tau}= \im (1+j\tau) \quad  JC_{j\s} = 
\im (1+j\sigma).
$$
Then $ JC_{\tau} \cap JC_{\s} = \Fix(j\sigma) \cap 
\Fix(j\tau) = \Fix (j\sigma, j\tau) = F$.
\end{proof}

\begin{thm}
The  Prym map 
$\cP^{iso}_3 :  \cRH_3^{iso} \ra  \cA_6^{(1,1,1,2,2,4)}$  is injective on the irreducible component consisting of Klein coverings of type II.2.
\end{thm}
\begin{proof}
Let $(P, \Xi)$ be an element on the image of $\cP^{ni}_3$. Then the group $\langle \sigma, \tau,j, -1 \rangle \simeq \ZZ_2^4$ acts on $P$ and according to Proposition \ref{isotyp_II.2}, this action  induces the isotypical decomposition 
$$
P =   P^*(C_{\s}/H) \boxplus P^*(C_{\tau}/H) \boxplus E_j \boxplus F
$$
with a distinguished elliptic curve $F$ of exponent $(4)$. 
Moreover, by construction 
$$
P^*(C_{\s}/H)+F \sim (JC_{j\s}, 2\Theta_{C_{j\s}}),  \quad  P^*(C_{\tau}/H)+F \sim (JC_{j\tau}, 2\Theta_{C_{j\tau}}).
$$
So the isotypical components determine uniquely the 
polarised Jacobians called $JC_{j\s}$ and $JD_{j\tau}$ and
by Torelli Theorem the genus 3 curves  $C_{j\s}$ and $ C_{j\tau}$ are also determined. Since $F$ is embedded in the Jacobians by Lemma \ref{key_lemma}, we consider the fibered product 
\begin{equation}\label{fibred_product_II.2}
\xymatrix@R=.9cm@C=1cm{
 \widetilde{C_{j\s} \times_{F}  C_{j\tau}}\ar[d]_{2:1} \ar[r]^{ \ 2:1} &  C_{j\s}  \ar[d]^{h_\s}  \\
C_{j\tau} \ar[r]_{h_\tau} &  F 
}
\end{equation}

Arguing as in the other case, we get a subdiagram of Klein covering diagram  
\begin{equation}\label{diag:MumII2}
\xymatrix@R=.9cm@C=.6cm{
& \tC \ar[dr] \ar[dl] \ar[d] & \\
C_{j\s}\ar[dr] &  C_{\s\tau} \ar[d] & C_{j\tau} \ar[dl]\\
&F& 
}
\end{equation}
hence $\tC=\widetilde{C_{j\s} \times_{F} C_{j\tau}}$. In particular, we can distinguish the subgroup $\left<j\s,j\tau\right>\subseteq \Aut(J\tC)$. 
Now, one can lift the hyperelliptic involution from $C_{j\s}$ to get two involutions $j', j'j\s$ with the sum of fixed points being equal to at least twice the number of Weierstrass points on $C_{j\s}$ that is $16$.
Hence the lifts have to be $j$ and $\s$. 
In par\-ti\-cu\-lar, the curve $\tC$ admits a distinguished group of automorphisms
$\langle j\s, j\tau, j \rangle \simeq \ZZ_2^3$. Therefore,
we can uniquely recover the Klein subgroup of fixed-point-free involutions $\left<\s,\tau\right>$ and hence the element $(H, \langle \eta, \xi \rangle) \in \cRH_{3}^{iso}$.
\end{proof}

\begin{rem}
Following Diagram \ref{diag:partial_tower_II2} and Proposition \ref{isotyp_II.2} consider the covering $\tC\to C_\s$. Since $JC_\s=JH^*\boxplus 
JD_\s$, we get that the Prym variety of the covering equals $P(\tC/C_\s)=JD^*_\tau\boxplus F\boxplus E$. The involution $j\s\tau$ is genuine and its fixed point set equals  $JC_{j\s\tau}\cap P(\tC/C_{\s})=JD^*_\tau\boxplus E_j$, hence is of codimension 1 in $P(\tC/C_{\s})$. This shows that $j\s\tau$ is a pseudoreflection of geometric origin. Moreover, since $g(C_{j\s\tau})=5$ the Prym variety is in the image of $\mathcal{P}_5(\mathcal{RF}_{5,3})$, see \cite{ALN} for more details.
    \end{rem}

\section{Non-isotropic Klein coverings of type I.1}
For the self-containment of the paper, we would like to recall the $I.1$ case from \cite{BO24}, namely when $\eta=w_1-w_2, \xi=w_2-w_3, \eta+\xi=w_1-w_3$. This is a 
special case of the construction made in \cite{BO24} because the top curve is hyperelliptic. We would like to state the main steps of injectivity of the Prym map (a little differently than in the original paper).
\begin{rem}\label{thmI.1}
 The moduli space of coverings of type $I.1$ is isomorphic to $\PP_{3,5}$.
    The lifts of hyperelliptic involutions have $20,4,4,4$ fixed points, so in particular $\tC$ is hyperelliptic with the hyperelliptic involution $j$. The isotypical decomposition of the Prym variety $P=P(\tC/H)$ with respect to $\left<\s,\tau,-1\right>$ is
    $$
  P  =   JC_{\s,j\tau} \boxplus JC_{\s\tau,j\tau} \boxplus JC_{\tau, j\s\tau}
$$ 
The sum $JC_{\s,j\tau} \boxplus JC_{\s\tau,j\tau}$ equals the Jacobian $JC_{j\tau}$, whereas the sum $JC_{\s,j\tau} \boxplus JC_{\tau,j\s\tau}$ equals $JC_{j\s\tau}$. In particular $JC_{j\tau}\cap JC_{j\s\tau}=JC_{\s,j\tau}$. Now, the curve $\tC$ is the (normalised) fibered product of $C_{j\tau}\times_{C_{\s,j\tau}} C_{j\s\tau}$ and the group $\left<\s,\tau\right>$ can be constructed by composing $j\s$ and $j\tau$ with the hyperelliptic involution $j$ on $\tC$. 
We leave the details to the reader.
\end{rem}
\section{Application of the results to Klein Prym maps}
This section is devoted to a global Klein Prym map of genus 3 curves, i.e. when a bottom curve is not necessarily hyperelliptic.
Let $\cR_3^{V_4}=\cR_3^{iso} \sqcup \cR_3^{ni}$ be a decomposition of a moduli of Klein coverings of genus 3 curves into isotropic and non-isotropic coverings. Observe that 
the forgetful map $\cR_3^{V_4} \ra \cM_3 $ is \'etale of degree $\frac{1}{6}{63 \choose 2}= 651$.
\begin{lem}
    The spaces $\cR_3^{iso}$ and $\cR_3^{ni}$ are irreducible.
\end{lem}
\begin{proof}
Let $\cA_3[2]$ be the moduli space of principally polarized abelian threefolds with a 2-level structure, i.e., with a symplectic basis of the set of $2$-torsion points, see \cite[Section 8.3.1]{BL}.
Then $J_3[2]\subset\cA_3[2]$ as a space of Jacobians with level structure is an open dense subset, hence irreducible. If $(\mu_i,\lambda_i)$ is a symplectic basis of $JC[2]$, then the forgetful maps
$$J_3[2]\ni [JC,(\mu_i,\lambda_i)]\to [C,\mu_1, \lambda_1]\in \cR_3^{ni}$$ 
$$J_3[2]\ni [JC,(\mu_i,\lambda_i)]\to [C,\mu_1, \mu_2]\in \cR_3^{iso}$$
are surjective, hence the codomains are irreducible.     
\end{proof}

\begin{thm}\label{FullKlein}
    The Klein Prym maps $$
\cP^{iso}_3 :  \cR_3^{iso} \ra  \cA_6^{(1,1,1,2,2,4)} \qquad \cP^{ni}_3 :  \cR_3^{ni} \ra  \cA_6^{(1,1,1,1,4,4)}
$$ are generically finite. Moreover they are of degree $1$ on images of coverings of hyperelliptic curves of types $I.2, II.1, II.2$.
\end{thm}
\begin{proof}
    Let us consider the non-isotropic case. By Theorems \ref{thmI.1} and \ref{thmI.2} the Prym map ${\cP^{ni}_3}_{|\cRH_3^{ni}}$ is injective on each component. Moreover, images of respective components are disjoint as the isotypical decomposition yields components of different dimensions. Hence, the image of the Prym map, being an irreducible variety and containing $2$ disjoint subvarieties of dimension $5$ has to be of dimension at least $6=\dim(\cR_3^{ni})$. This shows that the map is generically finite.

    Now, let $P$ be in the image of the Prym map of hyperelliptic curves of any of $I.2, II.1, II.2$ type. Then, Proposition \ref{autoofP} implies the group $G$ is isomorphic to $\ZZ_2^4$.
    If we start with a non-hyperelliptic curve, then there is no $j$, so the analogous Proposition will yield $\left<\s,\tau, -1\right>\cong\ZZ_2^3$. Similarly, for hyperelliptic coverings (case I.1), when the lift of the hyperelliptic involution is hyperelliptic, hence induces $-1$.
    This shows there is precisely one element in the preimage of $P$. 
\end{proof}

The above theorem is an evidence to a positive answer of the following question.
\begin{question}
Are the following Prym maps
$$
\cP^{iso}_3 :  \cR_3^{iso} \ra  \cA_6^{(1,1,1,2,2,4)} \qquad \cP^{ni}_3 :  \cR_3^{ni} \ra  \cA_6^{(1,1,1,1,4,4)}
$$
    injective?
\end{question}


\begin{thebibliography}{999999}


\bibitem{ALN} R. Auffarth, M. Lahoz, J. C. Naranjo, \emph{Pseudoreflections on Prym Varieties}, Preprint: \verb|arXiv:2412.04940|.





\bibitem{BL} 
Ch.~Birkenhake, H.~ Lange, 
\textit{Complex abelian varieties},
Second edition. Grundlehren der Mathematischen
Wissenschaften, 302. Springer-Verlag (2004).

\bibitem{BO19} 
P. Borówka, A. Ortega,  \textit{Klein coverings of genus 2  curves},
Trans. Amer. Math. Soc., {\bf 373} (3) 1885--1907, 2020.






\bibitem{BO24} 
P. Borówka, A. Ortega,  \textit{Involutions on 
hyperellitptic curves}, Annali Scuola Normale
Superiore - Classe di Scienze, 2024 (to appear), \verb|doi.org/10.2422/2036-2145.202309_039|.




\bibitem{BNOS}
P. Borówka, J.C. Naranjo,  A. Ortega, A. Shatsila, \textit{Prym maps of cyclic coverings of hyperelliptic curves}. Preprint:   arXiv:2510.01330v2



\bibitem{BS25} 
P. Borówka, A. Shatsila 
\textit{$\ZZ_3 \times \ZZ_3$ coverings  of genus 2  curves}, Proc. AMS. 2026 (to appear), \verb|doi.org/10.1090/proc/17548|.

  \bibitem{DO}
 I.~Dolgachev,
  \textit{Classical algebraic geometry. A modern view}.
  Cambridge University Press (2012).

\bibitem{LR} H. Lange, R. E. Rodriguez, {\it Decomposition of Jacobians by Prym varieties}, volume 2310 of Lecture Notes in Mathematics. Springer, 2022 

  \bibitem{M1}
  D. Mumford, \textit{Prym varieties I}, 
 Contributions to analysis (a collection of papers dedicated to Lipman Bers),
Academic Press, New York, 325–350,  1974. 

\bibitem{NO22} J.C. Naranjo, A. Ortega, \emph{Global Prym-Torelli for double coverings ramified in at least 6 points} 
J. Algebraic Geom. {\bf 31} (2022), no.2, 387-- 396.


  \bibitem{NOS24} J.C. Naranjo, A. Ortega, I. Spelta, {\it Cyclic coverings of genus 2 curves of Sophie Germain
type}. Forum Math. Sigma, {\bf 12} : Paper No. 64, 14, 2024.
\bibitem{NOPS25}
J. C. Naranjo, A. Ortega, G.P. Pirola, I. Spelta, {\it Simplicity of some Jacobians with
many automorphisms}. Algebr. Geom. {\bf 12}(6):869--887, 2025.

\bibitem{Pardini} R. Pardini. {\it Abelian covers of algebraic varieties}. J. Reine Angew. Math., {\bf 417} 191–213, 1991.

\bibitem{S25}
A.  Shatsila,  {\it On the Prym map of degree 4 cyclic covers of genus 2 curves}
Preprint:  arXiv:2508.20838.

\bibitem{Sho}
V.V. Shokurov,  {\it Prym varieties: theory and applications}. Math USSR Izvestiya Vol. 23, No.1(1984)

\bibitem{T}
M. Teixidor i Bigas, {\it For which Jacobi varieties is $\Sing \Theta$ reducible?}. J. Reine Angew. Math. 354, 141-149 (1984), 141--149.

    \end{thebibliography}
\end{document}